\documentclass{amsart}
\usepackage{amsfonts}

\newtheorem{theorem}{Theorem}
\theoremstyle{plain}

\newtheorem{remark}{Remark}

\numberwithin{equation}{section}
\input{tcilatex}

\begin{document}
\title[Improving Schwarz Inequality]{Improving Schwarz Inequality in Inner
Product Spaces}
\author[S. S. Dragomir]{Silvestru Sever Dragomir$^{1,2}$}
\address{$^{1}$Mathematics, College of Engineering \& Science\\
Victoria University, PO Box 14428\\
Melbourne City, MC 8001, Australia.}
\email{sever.dragomir@vu.edu.au}
\urladdr{http://rgmia.org/dragomir}
\address{$^{2}$DST-NRF Centre of Excellence \\
in the Mathematical and Statistical Sciences, School of Computer Science \&
Applied Mathematics, University of the Witwatersrand, Private Bag 3,
Johannesburg 2050, South Africa}

\begin{abstract}
Some improvements of the celebrated Schwarz inequality in complex inner
product spaces are given. Applications for $n$-tuples of complex numbers are
provided.
\end{abstract}

\maketitle

\section{Introduction}

Let $\left( H,\left\langle \cdot ,\cdot \right\rangle \right) $ be an inner
product space over the real or complex numbers field $\mathbb{K}$. The
following inequality is well known in literature as the \textit{Schwarz
inequality}%
\begin{equation}
\left\Vert x\right\Vert ^{2}\left\Vert y\right\Vert ^{2}\geq \left\vert
\left\langle x,y\right\rangle \right\vert ^{2}\text{ }  \label{S}
\end{equation}%
for any $x,$ $y\in H.$ The equality case holds in (\ref{S}) if and only if
there exists a constant $\lambda \in \mathbb{K}$ such that $x=\lambda y.$
This inequality can be written in an equivalent form as 
\begin{equation}
\left\Vert x\right\Vert \left\Vert y\right\Vert \geq \left\vert \left\langle
x,y\right\rangle \right\vert .  \label{SS}
\end{equation}

Assume that $P:H\rightarrow H$ is an orthogonal projection on $H$, namely,
it satisfies the condition $P^{2}=P=P^{\ast }$. We obviously have in the
operator order of $B(H)$, the Banach algebra of all linear bounded operators
on $H$, that $0\leq P\leq 1_{H}$.

In the recent paper \cite[Eq. (2.6)]{SSDBul} we established among others that%
\begin{equation}
\left\Vert x\right\Vert \left\Vert y\right\Vert \geq \left\langle
Px,x\right\rangle ^{1/2}\left\langle Py,y\right\rangle ^{1/2}+\left\vert
\left\langle x,y\right\rangle -\left\langle Px,y\right\rangle \right\vert 
\label{e.1.0}
\end{equation}%
for any $x,$ $y\in H.$ Since by the triangle inequality we have 
\begin{equation*}
\left\vert \left\langle x,y\right\rangle -\left\langle Px,y\right\rangle
\right\vert \geq \left\vert \left\langle x,y\right\rangle \right\vert
-\left\vert \left\langle Px,y\right\rangle \right\vert 
\end{equation*}%
and by the Schwarz inequality for nonnegative selfadjoint operators we have 
\begin{equation*}
\left\langle Px,x\right\rangle ^{1/2}\left\langle Py,y\right\rangle
^{1/2}\geq \left\vert \left\langle Px,y\right\rangle \right\vert 
\end{equation*}%
for any $x,$ $y\in H,$ then we get from (\ref{e.1.0}) the following
refinement of (\ref{SS})%
\begin{equation}
\left\Vert x\right\Vert \left\Vert y\right\Vert -\left\vert \left\langle
x,y\right\rangle \right\vert \geq \left\langle Px,x\right\rangle
^{1/2}\left\langle Py,y\right\rangle ^{1/2}-\left\vert \left\langle
Px,y\right\rangle \right\vert \geq 0  \label{e.1.0.1}
\end{equation}%
for any $x,$ $y\in H.$

In 1985 the author \cite{SSD} (see also \cite{DS} or \cite[p. 36]{SSDB2})
established the following inequality related to Schwarz inequality\textit{\ }%
\begin{equation}
\left( \left\Vert x\right\Vert ^{2}\left\Vert z\right\Vert ^{2}-\left\vert
\left\langle x,z\right\rangle \right\vert ^{2}\right) \left( \left\Vert
y\right\Vert ^{2}\left\Vert z\right\Vert ^{2}-\left\vert \left\langle
y,z\right\rangle \right\vert ^{2}\right) \geq \left\vert \left\langle
x,y\right\rangle \left\Vert z\right\Vert ^{2}-\left\langle x,z\right\rangle
\left\langle z,y\right\rangle \right\vert ^{2}  \label{S1}
\end{equation}%
for any $x,$ $y,$ $z\in H$ and obtained, as a consequence, the following
refinement of (\ref{SS}):%
\begin{equation}
\left\Vert x\right\Vert \left\Vert y\right\Vert \geq \left\vert \left\langle
x,y\right\rangle -\left\langle x,e\right\rangle \left\langle
e,y\right\rangle \right\vert +\left\vert \left\langle x,e\right\rangle
\left\langle e,y\right\rangle \right\vert \geq \left\vert \left\langle
x,y\right\rangle \right\vert   \label{RS}
\end{equation}%
for any $x,$ $y,$ $e\in H$ with $\left\Vert e\right\Vert =1.$

If we take the square root in (\ref{S1}) and use the triangle inequality, we
get for $x,$ $y,$ $z\in H\setminus \left\{ 0\right\} $ that%
\begin{align*}
& \left( \left\Vert x\right\Vert ^{2}\left\Vert z\right\Vert ^{2}-\left\vert
\left\langle x,z\right\rangle \right\vert ^{2}\right) ^{1/2}\left(
\left\Vert y\right\Vert ^{2}\left\Vert z\right\Vert ^{2}-\left\vert
\left\langle y,z\right\rangle \right\vert ^{2}\right) ^{1/2} \\
& \geq \left\vert \left\langle x,y\right\rangle \left\Vert z\right\Vert
^{2}-\left\langle x,z\right\rangle \left\langle z,y\right\rangle \right\vert
\geq \left\vert \left\langle x,z\right\rangle \left\langle z,y\right\rangle
\right\vert -\left\vert \left\langle x,y\right\rangle \right\vert \left\Vert
z\right\Vert ^{2}
\end{align*}%
which by division with $\left\Vert x\right\Vert ^{2}\left\Vert y\right\Vert
^{2}\left\Vert z\right\Vert ^{2}\neq 0$ produces%
\begin{equation}
\frac{\left\vert \left\langle x,y\right\rangle \right\vert }{\left\Vert
x\right\Vert \left\Vert y\right\Vert }\geq \frac{\left\vert \left\langle
x,z\right\rangle \right\vert }{\left\Vert x\right\Vert \left\Vert
z\right\Vert }\frac{\left\vert \left\langle z,y\right\rangle \right\vert }{%
\left\Vert z\right\Vert \left\Vert y\right\Vert }-\sqrt{1-\frac{\left\vert
\left\langle x,z\right\rangle \right\vert ^{2}}{\left\Vert x\right\Vert
^{2}\left\Vert z\right\Vert ^{2}}}\sqrt{1-\frac{\left\vert \left\langle
y,z\right\rangle \right\vert ^{2}}{\left\Vert y\right\Vert ^{2}\left\Vert
z\right\Vert ^{2}}}.  \label{e.1.1}
\end{equation}

If the angle between the vectors $x,$ $y,$ $\Psi _{xy}\in \left[ 0,\pi /2%
\right] ,$ is defined by \cite{L} 
\begin{equation}
\cos \Psi _{xy}=\frac{\left\vert \left\langle x,y\right\rangle \right\vert }{%
\left\Vert x\right\Vert \left\Vert y\right\Vert },\text{ }x,\text{ }y\neq 0,
\label{e.1.2}
\end{equation}%
then the function $\Psi _{xy}$ is a natural metric on complex projective
space, since is satisfies the inequality \cite{L}%
\begin{equation}
\Psi _{xy}\leq \Psi _{xz}+\Psi _{zy}\text{ for any }x,\text{ }y,\text{ }%
z\neq 0.  \label{e.1.3}
\end{equation}

By using (\ref{e.1.2}) we have by (\ref{e.1.1}) that 
\begin{equation*}
\cos \Psi _{xy}\geq \cos \Psi _{xz}\cos \Psi _{zx}-\sin \Psi _{xz}\sin \Psi
_{zx}=\cos \left( \Psi _{xz}+\Psi _{zx}\right) ,
\end{equation*}%
which is equivalent to (\ref{e.1.3}) since the function $\cos $ is
decreasing on $\left[ 0,\pi \right] .$ This provides a different proof of (%
\ref{e.1.3}) than the one from \cite{L} where it was done by utilising the
celebrated \textit{Kre\u{\i}n's inequality} \cite{K}, see also \cite[p. 56]%
{GR}, 
\begin{equation}
\Phi _{x,y}\leq \Phi _{x,z}+\Phi _{x,z}\text{ for any }x,\text{ }y,\text{ }%
z\neq 0,  \label{Kr}
\end{equation}%
obtained for angles $\Phi _{x,y}$ between two vectors $x,$ $y$, where in
this case $\Phi _{x,y}$ is defined by 
\begin{equation*}
\cos \Phi _{x,y}=\frac{\func{Re}\left\langle x,y\right\rangle }{\left\Vert
x\right\Vert \left\Vert y\right\Vert },\text{ }x,\text{ }y\neq 0.
\end{equation*}

The following inequality has been obtained by Wang and Zhang in \cite{WZ}
(see also \cite[p. 195]{Z})%
\begin{equation}
\sqrt{1-\frac{\left\vert \left\langle x,y\right\rangle \right\vert ^{2}}{%
\left\Vert x\right\Vert ^{2}\left\Vert y\right\Vert ^{2}}}\leq \sqrt{1-\frac{%
\left\vert \left\langle x,z\right\rangle \right\vert ^{2}}{\left\Vert
x\right\Vert ^{2}\left\Vert z\right\Vert ^{2}}}+\sqrt{1-\frac{\left\vert
\left\langle y,z\right\rangle \right\vert ^{2}}{\left\Vert y\right\Vert
^{2}\left\Vert z\right\Vert ^{2}}}  \label{e.1.4}
\end{equation}%
for any $x,$ $y,$ $z\in H\setminus \left\{ 0\right\} .$ Using the above
notations the inequality (\ref{e.1.4}) can be written as \cite{L} 
\begin{equation}
\sin \Psi _{x,y}\leq \sin \Psi _{x,z}+\sin \Psi _{x,z}  \label{e.1.5}
\end{equation}%
for any $x,$ $y,$ $z\in H\setminus \left\{ 0\right\} .$ It also provides
another triangle type inequality complementing the Kre\u{\i}n and Lin
inequalities above.

The corresponding result for the angle $\Phi _{x,y}$ was obtained by Lin in 
\cite{L} as%
\begin{equation}
\sin \Phi _{x,y}\leq \sin \Phi _{x,z}+\sin \Phi _{x,z},\text{ for any }x,%
\text{ }y,\text{ }z\neq 0,  \label{e.1.6}
\end{equation}%
or, equivalently, as 
\begin{equation}
\sqrt{1-\frac{\left\vert \func{Re}\left\langle x,y\right\rangle \right\vert
^{2}}{\left\Vert x\right\Vert ^{2}\left\Vert y\right\Vert ^{2}}}\leq \sqrt{1-%
\frac{\left\vert \func{Re}\left\langle x,z\right\rangle \right\vert ^{2}}{%
\left\Vert x\right\Vert ^{2}\left\Vert z\right\Vert ^{2}}}+\sqrt{1-\frac{%
\left\vert \func{Re}\left\langle y,z\right\rangle \right\vert ^{2}}{%
\left\Vert y\right\Vert ^{2}\left\Vert z\right\Vert ^{2}}}  \label{e.1.7}
\end{equation}%
for any $x,$ $y,$ $z\neq 0.$

In \cite{L} the author has also shown that, in fact, the inequalities (\ref%
{e.1.4}) and (\ref{e.1.7}) can be extended for any power $p>2,$ namely as 
\begin{equation}
\left( 1-\frac{\left\vert \left\langle x,y\right\rangle \right\vert ^{p}}{%
\left\Vert x\right\Vert ^{p}\left\Vert y\right\Vert ^{p}}\right) ^{1/p}\leq
\left( 1-\frac{\left\vert \left\langle x,z\right\rangle \right\vert ^{p}}{%
\left\Vert x\right\Vert ^{p}\left\Vert z\right\Vert ^{p}}\right)
^{1/p}+\left( 1-\frac{\left\vert \left\langle y,z\right\rangle \right\vert
^{p}}{\left\Vert y\right\Vert ^{p}\left\Vert z\right\Vert ^{p}}\right) ^{1/p}
\label{e.1.8}
\end{equation}%
and%
\begin{equation}
\left( 1-\frac{\left\vert \func{Re}\left\langle x,y\right\rangle \right\vert
^{p}}{\left\Vert x\right\Vert ^{p}\left\Vert y\right\Vert ^{p}}\right)
^{1/p}\leq \left( 1-\frac{\left\vert \func{Re}\left\langle x,z\right\rangle
\right\vert ^{p}}{\left\Vert x\right\Vert ^{p}\left\Vert z\right\Vert ^{p}}%
\right) ^{1/p}+\left( 1-\frac{\left\vert \func{Re}\left\langle
y,z\right\rangle \right\vert ^{p}}{\left\Vert y\right\Vert ^{p}\left\Vert
z\right\Vert ^{p}}\right) ^{1/p}  \label{e.1.9}
\end{equation}%
for any $x,$ $y,$ $z\neq 0.$

In this paper we obtain some improvements of Schwarz inequality in complex
inner product spaces as follows. For various inequalities related to this
famous result see the monographs \cite{SSDB1} and \cite{SSDB2}.

\section{Main Results}

Employing Lin's inequalities (\ref{e.1.8}) and (\ref{e.1.9}) we can obtain
the following refinements of Schwarz's inequality.

\begin{theorem}
\label{t.2.0}Let $x,$ $y,$ $e\in H$ with $\left\Vert e\right\Vert =1$ and $%
p\geq 2.$ Then we have the following refinements of Schwarz inequality%
\begin{equation}
\left\Vert x\right\Vert ^{p}\left\Vert y\right\Vert ^{p}-\left\vert
\left\langle x,y\right\rangle \right\vert ^{p}\geq \left( \det \left[ 
\begin{array}{ccc}
\left\Vert x\right\Vert &  & \left( \left\Vert x\right\Vert ^{p}-\left\vert
\left\langle x,e\right\rangle \right\vert ^{p}\right) ^{1/p} \\ 
&  &  \\ 
\left\Vert y\right\Vert &  & \left( \left\Vert y\right\Vert ^{p}-\left\vert
\left\langle y,e\right\rangle \right\vert ^{p}\right) ^{1/p}%
\end{array}%
\right] \right) ^{p}  \label{e.1.10}
\end{equation}%
and%
\begin{equation}
\left\Vert x\right\Vert ^{p}\left\Vert y\right\Vert ^{p}-\left\vert \func{Re}%
\left\langle x,y\right\rangle \right\vert ^{p}\geq \left( \det \left[ 
\begin{array}{ccc}
\left\Vert x\right\Vert &  & \left( \left\Vert x\right\Vert ^{p}-\left\vert 
\func{Re}\left\langle x,e\right\rangle \right\vert ^{p}\right) ^{1/p} \\ 
&  &  \\ 
\left\Vert y\right\Vert &  & \left( \left\Vert y\right\Vert ^{p}-\left\vert 
\func{Re}\left\langle y,e\right\rangle \right\vert ^{p}\right) ^{1/p}%
\end{array}%
\right] \right) ^{p}.  \label{e.1.11}
\end{equation}
\end{theorem}

\begin{proof}
We observe that, by (\ref{e.1.8}) and (\ref{e.1.9}) 
\begin{equation*}
d_{p}\left( x,y\right) :=\left( 1-\frac{\left\vert \left\langle
x,y\right\rangle \right\vert ^{p}}{\left\Vert x\right\Vert ^{p}\left\Vert
y\right\Vert ^{p}}\right) ^{1/p}\text{ and }\delta _{p}\left( x,y\right)
:=\left( 1-\frac{\left\vert \func{Re}\left\langle x,y\right\rangle
\right\vert ^{p}}{\left\Vert x\right\Vert ^{p}\left\Vert y\right\Vert ^{p}}%
\right) ^{1/p}\text{ }
\end{equation*}%
are distances and by the continuity property of the distance $d,$ namely%
\begin{equation*}
\left\vert d\left( x,z\right) -d\left( y,z\right) \right\vert \leq d\left(
x,y\right)
\end{equation*}%
we get 
\begin{equation}
\left\vert \left( 1-\frac{\left\vert \left\langle x,e\right\rangle
\right\vert ^{p}}{\left\Vert x\right\Vert ^{p}}\right) ^{1/p}-\left( 1-\frac{%
\left\vert \left\langle y,e\right\rangle \right\vert ^{p}}{\left\Vert
y\right\Vert ^{p}}\right) ^{1/p}\right\vert \leq \left( 1-\frac{\left\vert
\left\langle x,y\right\rangle \right\vert ^{p}}{\left\Vert x\right\Vert
^{p}\left\Vert y\right\Vert ^{p}}\right) ^{1/p}  \label{e.1.12}
\end{equation}%
and%
\begin{equation}
\left\vert \left( 1-\frac{\left\vert \func{Re}\left\langle x,e\right\rangle
\right\vert ^{p}}{\left\Vert x\right\Vert ^{p}}\right) ^{1/p}-\left( 1-\frac{%
\left\vert \func{Re}\left\langle y,e\right\rangle \right\vert ^{p}}{%
\left\Vert y\right\Vert ^{p}}\right) ^{1/p}\right\vert \leq \left( 1-\frac{%
\left\vert \func{Re}\left\langle x,y\right\rangle \right\vert ^{p}}{%
\left\Vert x\right\Vert ^{p}\left\Vert y\right\Vert ^{p}}\right) ^{1/p}
\label{e.1.13}
\end{equation}%
for any $x,$ $y\neq 0$ and $e\in H$ with $\left\Vert e\right\Vert =1.$

If we take the power $p$ in (\ref{e.1.12}) and (\ref{e.1.13}) and multiply
with $\left\Vert x\right\Vert ^{p}\left\Vert y\right\Vert ^{p}>0,$ then we
get the desired results (\ref{e.1.10}) and (\ref{e.1.11}).
\end{proof}

The following similar result can be stated as well:

\begin{theorem}
\label{t.2.1}Let $x,$ $y,$ $e\in H$ with $\left\Vert e\right\Vert =1.$ Then
we have the following refinement of Schwarz inequality%
\begin{equation}
\left\Vert x\right\Vert ^{2}\left\Vert y\right\Vert ^{2}-\left\vert
\left\langle x,y\right\rangle \right\vert ^{2}\geq \left( \det \left[ 
\begin{array}{ccc}
\left\vert \left\langle x,e\right\rangle \right\vert &  & \left( \left\Vert
x\right\Vert ^{2}-\left\vert \left\langle x,e\right\rangle \right\vert
^{2}\right) ^{1/2} \\ 
&  &  \\ 
\left\vert \left\langle y,e\right\rangle \right\vert &  & \left( \left\Vert
y\right\Vert ^{2}-\left\vert \left\langle y,e\right\rangle \right\vert
^{2}\right) ^{1/2}%
\end{array}%
\right] \right) ^{2}.  \label{e.2.0}
\end{equation}
\end{theorem}

\begin{proof}
We have by Schwarz's inequality that%
\begin{equation}
\left\vert \left\langle x-\alpha e,y-\overline{\beta }e\right\rangle
\right\vert ^{2}\leq \left\Vert x-\alpha e\right\Vert ^{2}\left\Vert y-%
\overline{\beta }e\right\Vert ^{2}  \label{e.2.1}
\end{equation}%
for any $x,$ $y,$ $e\in H$ and $\alpha ,$ $\beta \in \mathbb{C}.$

Since $\left\Vert e\right\Vert =1,$ then%
\begin{equation}
\left\langle x-\alpha e,y-\overline{\beta }e\right\rangle =\left\langle
x,y\right\rangle -\alpha \left\langle e,y\right\rangle -\beta \left\langle
x,e\right\rangle +\alpha \beta ,  \label{e.2.2}
\end{equation}%
\begin{equation*}
\left\Vert x-\alpha e\right\Vert ^{2}=\left\Vert x\right\Vert ^{2}-\left( 2%
\func{Re}\left[ \overline{\alpha }\left\langle x,e\right\rangle \right]
-\left\vert \alpha \right\vert ^{2}\right)
\end{equation*}%
and 
\begin{equation*}
\left\Vert y-\overline{\beta }e\right\Vert ^{2}=\left\Vert y\right\Vert
^{2}-\left( 2\func{Re}\left[ \beta \left\langle y,e\right\rangle \right]
-\left\vert \beta \right\vert ^{2}\right) .
\end{equation*}%
This implies that 
\begin{align}
& \left\Vert x-\alpha e\right\Vert ^{2}\left\Vert y-\overline{\beta }%
e\right\Vert ^{2}  \label{e.2.3} \\
& =\left[ \left\Vert x\right\Vert ^{2}-\left( 2\func{Re}\left[ \overline{%
\alpha }\left\langle x,e\right\rangle \right] -\left\vert \alpha \right\vert
^{2}\right) \right] \left[ \left\Vert y\right\Vert ^{2}-\left( 2\func{Re}%
\left[ \beta \left\langle x,e\right\rangle \right] -\left\vert \beta
\right\vert ^{2}\right) \right]  \notag \\
& =\left\Vert x\right\Vert ^{2}\left\Vert y\right\Vert ^{2}-\left\Vert
y\right\Vert ^{2}\left( 2\func{Re}\left[ \overline{\alpha }\left\langle
x,e\right\rangle \right] -\left\vert \alpha \right\vert ^{2}\right)
-\left\Vert x\right\Vert ^{2}\left( 2\func{Re}\left[ \beta \left\langle
y,e\right\rangle \right] -\left\vert \beta \right\vert ^{2}\right)  \notag \\
& +\left( 2\func{Re}\left[ \overline{\alpha }\left\langle x,e\right\rangle %
\right] -\left\vert \alpha \right\vert ^{2}\right) \left( 2\func{Re}\left[
\beta \left\langle y,e\right\rangle \right] -\left\vert \beta \right\vert
^{2}\right)  \notag \\
& =\left\Vert x\right\Vert ^{2}\left\Vert y\right\Vert ^{2}-\left(
\left\Vert y\right\Vert ^{2}-\left\vert \left\langle y,e\right\rangle
\right\vert ^{2}\right) \left( 2\func{Re}\left[ \overline{\alpha }%
\left\langle x,e\right\rangle \right] -\left\vert \alpha \right\vert
^{2}\right)  \notag \\
& -\left( \left\Vert x\right\Vert ^{2}-\left\vert \left\langle
x,e\right\rangle \right\vert ^{2}\right) \left( 2\func{Re}\left[ \beta
\left\langle y,e\right\rangle \right] -\left\vert \beta \right\vert
^{2}\right)  \notag \\
& -\left\vert \left\langle y,e\right\rangle \right\vert ^{2}\left( 2\func{Re}%
\left[ \overline{\alpha }\left\langle x,e\right\rangle \right] -\left\vert
\alpha \right\vert ^{2}\right) -\left\vert \left\langle x,e\right\rangle
\right\vert ^{2}\left( 2\func{Re}\left[ \beta \left\langle y,e\right\rangle %
\right] -\left\vert \beta \right\vert ^{2}\right)  \notag \\
& +\left( 2\func{Re}\left[ \overline{\alpha }\left\langle x,e\right\rangle %
\right] -\left\vert \alpha \right\vert ^{2}\right) \left( 2\func{Re}\left[
\beta \left\langle y,e\right\rangle \right] -\left\vert \beta \right\vert
^{2}\right)  \notag \\
& +\left\vert \left\langle y,e\right\rangle \right\vert ^{2}\left\vert
\left\langle x,e\right\rangle \right\vert ^{2}-\left\vert \left\langle
y,e\right\rangle \right\vert ^{2}\left\vert \left\langle x,e\right\rangle
\right\vert ^{2}  \notag \\
& =\left\Vert x\right\Vert ^{2}\left\Vert y\right\Vert ^{2}-\left(
\left\Vert y\right\Vert ^{2}-\left\vert \left\langle y,e\right\rangle
\right\vert ^{2}\right) \left( 2\func{Re}\left[ \overline{\alpha }%
\left\langle x,e\right\rangle \right] -\left\vert \alpha \right\vert
^{2}\right)  \notag \\
& -\left( \left\Vert x\right\Vert ^{2}-\left\vert \left\langle
x,e\right\rangle \right\vert ^{2}\right) \left( 2\func{Re}\left[ \beta
\left\langle y,e\right\rangle \right] -\left\vert \beta \right\vert
^{2}\right)  \notag \\
& +\left[ \left\vert \left\langle x,e\right\rangle \right\vert ^{2}-\left( 2%
\func{Re}\left[ \overline{\alpha }\left\langle x,e\right\rangle \right]
-\left\vert \alpha \right\vert ^{2}\right) \right] \left[ \left\vert
\left\langle y,e\right\rangle \right\vert ^{2}-\left( 2\func{Re}\left[ \beta
\left\langle y,e\right\rangle \right] -\left\vert \beta \right\vert
^{2}\right) \right]  \notag \\
& -\left\vert \left\langle y,e\right\rangle \right\vert ^{2}\left\vert
\left\langle x,e\right\rangle \right\vert ^{2}.  \notag
\end{align}%
Observe that 
\begin{equation*}
\left\vert \left\langle x,e\right\rangle \right\vert ^{2}-\left( 2\func{Re}%
\left[ \overline{\alpha }\left\langle x,e\right\rangle \right] -\left\vert
\alpha \right\vert ^{2}\right) =\left\vert \left\langle x,e\right\rangle
-\alpha \right\vert ^{2}
\end{equation*}%
and%
\begin{equation*}
\left\vert \left\langle y,e\right\rangle \right\vert ^{2}-\left( 2\func{Re}%
\left[ \beta \left\langle y,e\right\rangle \right] -\left\vert \beta
\right\vert ^{2}\right) =\left\vert \left\langle y,e\right\rangle -\overline{%
\beta }\right\vert ^{2}.
\end{equation*}%
Therefore, by (\ref{e.2.3}) we get%
\begin{align}
& \left\Vert x-\alpha e\right\Vert ^{2}\left\Vert y-\overline{\beta }%
e\right\Vert ^{2}  \label{e.2.3.1} \\
& =\left\Vert x\right\Vert ^{2}\left\Vert y\right\Vert ^{2}-\left(
\left\Vert y\right\Vert ^{2}-\left\vert \left\langle y,e\right\rangle
\right\vert ^{2}\right) \left( 2\func{Re}\left[ \overline{\alpha }%
\left\langle x,e\right\rangle \right] -\left\vert \alpha \right\vert
^{2}\right)  \notag \\
& -\left( \left\Vert x\right\Vert ^{2}-\left\vert \left\langle
x,e\right\rangle \right\vert ^{2}\right) \left( 2\func{Re}\left[ \beta
\left\langle y,e\right\rangle \right] -\left\vert \beta \right\vert
^{2}\right)  \notag \\
& +\left\vert \left\langle x,e\right\rangle -\alpha \right\vert
^{2}\left\vert \left\langle y,e\right\rangle -\overline{\beta }\right\vert
^{2}-\left\vert \left\langle y,e\right\rangle \right\vert ^{2}\left\vert
\left\langle x,e\right\rangle \right\vert ^{2}.  \notag
\end{align}%
Let $\alpha \in \mathbb{C}$ with $\alpha \neq \left\langle x,e\right\rangle $
and put%
\begin{equation}
\beta :=\frac{\alpha \overline{\left\langle y,e\right\rangle }}{\alpha
-\left\langle x,e\right\rangle }.  \label{e.2.4}
\end{equation}%
Then%
\begin{align*}
\left\vert \left\langle x,e\right\rangle -\alpha \right\vert ^{2}\left\vert
\left\langle y,e\right\rangle -\overline{\beta }\right\vert ^{2}&
=\left\vert \left\langle x,e\right\rangle -\alpha \right\vert ^{2}\left\vert
\left\langle y,e\right\rangle -\overline{\left( \frac{\alpha \overline{%
\left\langle y,e\right\rangle }}{\alpha -\left\langle x,e\right\rangle }%
\right) }\right\vert ^{2} \\
& =\left\vert \left\langle x,e\right\rangle -\alpha \right\vert
^{2}\left\vert \left\langle y,e\right\rangle -\frac{\overline{\alpha }%
\left\langle y,e\right\rangle }{\overline{\alpha }-\overline{\left\langle
x,e\right\rangle }}\right\vert ^{2} \\
& =\left\vert \left\langle y,e\right\rangle \right\vert ^{2}\left\vert
\left\langle x,e\right\rangle -\alpha \right\vert ^{2}\left\vert \frac{%
\overline{\left\langle x,e\right\rangle }}{\overline{\alpha }-\overline{%
\left\langle x,e\right\rangle }}\right\vert ^{2} \\
& =\left\vert \left\langle y,e\right\rangle \right\vert ^{2}\left\vert
\left\langle x,e\right\rangle \right\vert ^{2}
\end{align*}%
and%
\begin{equation*}
\alpha \beta =\alpha \left\langle e,y\right\rangle +\beta \left\langle
x,e\right\rangle .
\end{equation*}%
For these choices of $\alpha $ and $\beta $ we have by (\ref{e.2.1})-(\ref%
{e.2.3.1}) that%
\begin{align}
\left\vert \left\langle x,y\right\rangle \right\vert ^{2}& \leq \left\Vert
x\right\Vert ^{2}\left\Vert y\right\Vert ^{2}-\left( \left\Vert y\right\Vert
^{2}-\left\vert \left\langle y,e\right\rangle \right\vert ^{2}\right) \left(
2\func{Re}\left[ \overline{\alpha }\left\langle x,e\right\rangle \right]
-\left\vert \alpha \right\vert ^{2}\right)  \label{e.2.5} \\
& -\left( \left\Vert x\right\Vert ^{2}-\left\vert \left\langle
x,e\right\rangle \right\vert ^{2}\right) \left( 2\func{Re}\left[ \beta
\left\langle y,e\right\rangle \right] -\left\vert \beta \right\vert
^{2}\right) .  \notag
\end{align}%
By (\ref{e.2.4}) we also have 
\begin{equation}
2\func{Re}\left[ \beta \left\langle y,e\right\rangle \right] -\left\vert
\beta \right\vert ^{2}=\left\vert \left\langle y,e\right\rangle \right\vert
^{2}\left[ 2\func{Re}\left[ \frac{\alpha }{\alpha -\left\langle
x,e\right\rangle }\right] -\left\vert \frac{\alpha }{\alpha -\left\langle
x,e\right\rangle }\right\vert ^{2}\right] .  \label{e.2.5.1}
\end{equation}%
Take 
\begin{equation}
\alpha =\left\langle x,e\right\rangle +t\text{ with }t\in \mathbb{R},\text{ }%
t\neq 0.  \label{e.2.5.2}
\end{equation}%
Then by (\ref{e.2.5.1}) we have 
\begin{align*}
B\left( x,y,e,t\right) & :=\left( \left\Vert y\right\Vert ^{2}-\left\vert
\left\langle y,e\right\rangle \right\vert ^{2}\right) \left( 2\func{Re}\left[
\overline{\alpha }\left\langle x,e\right\rangle \right] -\left\vert \alpha
\right\vert ^{2}\right) \\
& +\left( \left\Vert x\right\Vert ^{2}-\left\vert \left\langle
x,e\right\rangle \right\vert ^{2}\right) \left( 2\func{Re}\left[ \beta
\left\langle y,e\right\rangle \right] -\left\vert \beta \right\vert
^{2}\right) \\
& =\left( \left\Vert y\right\Vert ^{2}-\left\vert \left\langle
y,e\right\rangle \right\vert ^{2}\right) \left( 2\func{Re}\left[ \left( 
\overline{\left\langle x,e\right\rangle +t}\right) \left\langle
x,e\right\rangle \right] -\left\vert \left\langle x,e\right\rangle
+t\right\vert ^{2}\right) \\
& +\left( \left\Vert x\right\Vert ^{2}-\left\vert \left\langle
x,e\right\rangle \right\vert ^{2}\right) \left\vert \left\langle
y,e\right\rangle \right\vert ^{2}\left[ 2\func{Re}\left[ \frac{\left\langle
x,e\right\rangle +t}{t}\right] -\left\vert \frac{\left\langle
x,e\right\rangle +t}{t}\right\vert ^{2}\right] .
\end{align*}%
Since 
\begin{align*}
& 2\func{Re}\left[ \left( \overline{\left\langle x,e\right\rangle +t}\right)
\left\langle x,e\right\rangle \right] -\left\vert \left\langle
x,e\right\rangle +t\right\vert ^{2} \\
& =2\func{Re}\left[ \left\vert \left\langle x,e\right\rangle \right\vert
^{2}+t\left\langle x,e\right\rangle \right] -\left\vert \left\langle
x,e\right\rangle \right\vert ^{2}-2t\func{Re}\left\langle x,e\right\rangle
-t^{2} \\
& =\left\vert \left\langle x,e\right\rangle \right\vert ^{2}-t^{2}
\end{align*}%
and%
\begin{align*}
& 2\func{Re}\left[ \frac{\left\langle x,e\right\rangle +t}{t}\right]
-\left\vert \frac{\left\langle x,e\right\rangle +t}{t}\right\vert ^{2} \\
& =2\func{Re}\left[ \frac{\left\langle x,e\right\rangle }{t}+1\right]
-\left\vert \frac{\left\langle x,e\right\rangle }{t}+1\right\vert ^{2} \\
& =\frac{2\func{Re}\left\langle x,e\right\rangle }{t}+2-\frac{\left\vert
\left\langle x,e\right\rangle \right\vert ^{2}}{t^{2}}-\frac{2\func{Re}%
\left\langle x,e\right\rangle }{t}-1=1-\frac{\left\vert \left\langle
x,e\right\rangle \right\vert ^{2}}{t^{2}},
\end{align*}%
then we get%
\begin{align*}
B\left( x,y,e,t\right) & =\left( \left\Vert y\right\Vert ^{2}-\left\vert
\left\langle y,e\right\rangle \right\vert ^{2}\right) \left( \left\vert
\left\langle x,e\right\rangle \right\vert ^{2}-t^{2}\right) \\
& +\left( \left\Vert x\right\Vert ^{2}-\left\vert \left\langle
x,e\right\rangle \right\vert ^{2}\right) \left\vert \left\langle
y,e\right\rangle \right\vert ^{2}\left( 1-\frac{\left\vert \left\langle
x,e\right\rangle \right\vert ^{2}}{t^{2}}\right) \\
& =\left( \left\Vert y\right\Vert ^{2}-\left\vert \left\langle
y,e\right\rangle \right\vert ^{2}\right) \left( \left\vert \left\langle
x,e\right\rangle \right\vert ^{2}-t^{2}\right) \\
& -\left( \left\Vert x\right\Vert ^{2}-\left\vert \left\langle
x,e\right\rangle \right\vert ^{2}\right) \left\vert \left\langle
y,e\right\rangle \right\vert ^{2}\left( \frac{\left\vert \left\langle
x,e\right\rangle \right\vert ^{2}-t^{2}}{t^{2}}\right) \\
& =\left( \left\vert \left\langle x,e\right\rangle \right\vert
^{2}-t^{2}\right) \left[ \left\Vert y\right\Vert ^{2}-\left\vert
\left\langle y,e\right\rangle \right\vert ^{2}-\frac{\left( \left\Vert
x\right\Vert ^{2}-\left\vert \left\langle x,e\right\rangle \right\vert
^{2}\right) \left\vert \left\langle y,e\right\rangle \right\vert ^{2}}{t^{2}}%
\right]
\end{align*}%
for $t\in \mathbb{R},$ $t\neq 0.$

Assume that $\left\langle x,e\right\rangle ,$ $\left\langle y,e\right\rangle
\neq 0$ and $\left\Vert x\right\Vert \neq \left\vert \left\langle
x,e\right\rangle \right\vert ,$ $\left\Vert y\right\Vert \neq \left\vert
\left\langle y,e\right\rangle \right\vert .$

If we take $t=t_{0}\neq 0$ with%
\begin{equation*}
t_{0}^{2}=\left\vert \left\langle x,e\right\rangle \left\langle
y,e\right\rangle \right\vert \sqrt{\frac{\left\Vert x\right\Vert
^{2}-\left\vert \left\langle x,e\right\rangle \right\vert ^{2}}{\left\Vert
y\right\Vert ^{2}-\left\vert \left\langle y,e\right\rangle \right\vert ^{2}}}
\end{equation*}%
then we get 
\begin{align*}
& B\left( x,y,e,t_{0}\right) \\
& =\left( \left\vert \left\langle x,e\right\rangle \right\vert
^{2}-\left\vert \left\langle x,e\right\rangle \left\langle y,e\right\rangle
\right\vert \sqrt{\frac{\left\Vert x\right\Vert ^{2}-\left\vert \left\langle
x,e\right\rangle \right\vert ^{2}}{\left\Vert y\right\Vert ^{2}-\left\vert
\left\langle y,e\right\rangle \right\vert ^{2}}}\right) \\
& \times \left[ \left\Vert y\right\Vert ^{2}-\left\vert \left\langle
y,e\right\rangle \right\vert ^{2}-\frac{\left( \left\Vert x\right\Vert
^{2}-\left\vert \left\langle x,e\right\rangle \right\vert ^{2}\right)
\left\vert \left\langle y,e\right\rangle \right\vert ^{2}}{\left\vert
\left\langle x,e\right\rangle \left\langle y,e\right\rangle \right\vert 
\sqrt{\frac{\left\Vert x\right\Vert ^{2}-\left\vert \left\langle
x,e\right\rangle \right\vert ^{2}}{\left\Vert y\right\Vert ^{2}-\left\vert
\left\langle y,e\right\rangle \right\vert ^{2}}}}\right] \\
& =\frac{\left\vert \left\langle x,e\right\rangle \right\vert }{\sqrt{%
\left\Vert y\right\Vert ^{2}-\left\vert \left\langle y,e\right\rangle
\right\vert ^{2}}}\left( \left\vert \left\langle x,e\right\rangle
\right\vert \sqrt{\left\Vert y\right\Vert ^{2}-\left\vert \left\langle
y,e\right\rangle \right\vert ^{2}}-\left\vert \left\langle y,e\right\rangle
\right\vert \sqrt{\left\Vert x\right\Vert ^{2}-\left\vert \left\langle
x,e\right\rangle \right\vert ^{2}}\right) \\
& \times \left[ \frac{\left\vert \left\langle x,e\right\rangle \right\vert 
\sqrt{\left\Vert y\right\Vert ^{2}-\left\vert \left\langle y,e\right\rangle
\right\vert ^{2}}-\sqrt{\left\Vert x\right\Vert ^{2}-\left\vert \left\langle
x,e\right\rangle \right\vert ^{2}}\left\vert \left\langle y,e\right\rangle
\right\vert }{\left\vert \left\langle x,e\right\rangle \right\vert }\right] 
\sqrt{\left\Vert y\right\Vert ^{2}-\left\vert \left\langle y,e\right\rangle
\right\vert ^{2}} \\
& =\left( \left\vert \left\langle x,e\right\rangle \right\vert \sqrt{%
\left\Vert y\right\Vert ^{2}-\left\vert \left\langle y,e\right\rangle
\right\vert ^{2}}-\left\vert \left\langle y,e\right\rangle \right\vert \sqrt{%
\left\Vert x\right\Vert ^{2}-\left\vert \left\langle x,e\right\rangle
\right\vert ^{2}}\right) ^{2}.
\end{align*}%
By using the inequality (\ref{e.2.5}) we then have 
\begin{align*}
\left\vert \left\langle x,y\right\rangle \right\vert ^{2}& \leq \left\Vert
x\right\Vert ^{2}\left\Vert y\right\Vert ^{2}-B\left( x,y,e,t_{0}\right) \\
& =\left\Vert x\right\Vert ^{2}\left\Vert y\right\Vert ^{2}-\left(
\left\vert \left\langle x,e\right\rangle \right\vert \sqrt{\left\Vert
y\right\Vert ^{2}-\left\vert \left\langle y,e\right\rangle \right\vert ^{2}}%
-\left\vert \left\langle y,e\right\rangle \right\vert \sqrt{\left\Vert
x\right\Vert ^{2}-\left\vert \left\langle x,e\right\rangle \right\vert ^{2}}%
\right) ^{2},
\end{align*}%
which proves the desired result (\ref{e.2.0}).

Now, if $\left\langle x,e\right\rangle =0$ i.e. $x\perp e$, then (\ref{e.2.0}%
) becomes 
\begin{equation*}
\left\vert \left\langle y,e\right\rangle \right\vert ^{2}\left\Vert
x\right\Vert ^{2}+\left\vert \left\langle x,y\right\rangle \right\vert
^{2}\leq \left\Vert x\right\Vert ^{2}\left\Vert y\right\Vert ^{2}
\end{equation*}%
which is trivial for $x=0$ and becomes the Bessel's inequality 
\begin{equation*}
\left\vert \left\langle y,e\right\rangle \right\vert ^{2}+\left\vert
\left\langle y,\frac{x}{\left\Vert x\right\Vert }\right\rangle \right\vert
^{2}\leq \left\Vert y\right\Vert ^{2}
\end{equation*}%
for the orthonormal family $\left\{ e,\frac{x}{\left\Vert x\right\Vert }%
\right\} .$

A similar argument applies for $\left\langle y,e\right\rangle =0.$

Also, if $\left\Vert x\right\Vert ^{2}=\left\vert \left\langle
x,e\right\rangle \right\vert ^{2}$ then by the equality case in Schwarz
inequality for the vectors $x$ and $e$ we get that there exists a constant $%
\gamma $ such that $x=\gamma e.$ In this situation (\ref{e.2.0}) becomes an
equality.

A similar argument applies if $\left\Vert y\right\Vert ^{2}=\left\vert
\left\langle y,e\right\rangle \right\vert ^{2}.$
\end{proof}

\begin{remark}
\label{r.2.1}If $\left( H,\left\langle \cdot ,\cdot \right\rangle \right) $
is a complex inner product space, then $\left( H,\left\langle \cdot ,\cdot
\right\rangle _{r}\right) $ with 
\begin{equation*}
\left\langle x,y\right\rangle _{r}:=\func{Re}\left\langle x,y\right\rangle
\end{equation*}%
is a real inner product space and $\left\langle x,x\right\rangle
^{1/2}=\left\langle x,x\right\rangle _{r}^{1/2}=\left\Vert x\right\Vert $
for $x\in H.$ Therefore by (\ref{e.2.0}) for $\left\langle \cdot ,\cdot
\right\rangle _{r}$ we get%
\begin{multline}
\left\Vert x\right\Vert ^{2}\left\Vert y\right\Vert ^{2}-\left\vert \func{Re}%
\left\langle x,y\right\rangle \right\vert ^{2}  \label{e.2.6} \\
\geq \left( \det \left[ 
\begin{array}{ccc}
\left\vert \func{Re}\left\langle x,e\right\rangle \right\vert &  & \left(
\left\Vert x\right\Vert ^{2}-\left\vert \func{Re}\left\langle
x,e\right\rangle \right\vert ^{2}\right) ^{1/2} \\ 
&  &  \\ 
\left\vert \func{Re}\left\langle y,e\right\rangle \right\vert &  & \left(
\left\Vert y\right\Vert ^{2}-\left\vert \func{Re}\left\langle
y,e\right\rangle \right\vert ^{2}\right) ^{1/2}%
\end{array}%
\right] \right) ^{2}
\end{multline}%
for any $x,$ $y,$ $e\in H$ with $\left\Vert e\right\Vert =1.$
\end{remark}

\section{An Application for $n$-Tuples of Complex Numbers}

Let $x=\left( x_{1},...,x_{n}\right) ,$ $y=\left( y_{1},...,y_{n}\right) ,$ $%
e=\left( e_{1},...,e_{n}\right) \in \mathbb{C}^{n}$ with $%
\sum_{k=1}^{n}\left\vert e_{k}\right\vert ^{2}=1.$ Then by writing the above
inequalities (\ref{e.1.10}) and (\ref{e.2.0}) for the inner product $%
\left\langle x,y\right\rangle :=\sum_{k=1}^{n}x_{k}\overline{y}_{k}$ we have
for $p\geq 2,$ that%
\begin{multline}
\left( \sum_{k=1}^{n}\left\vert x_{k}\right\vert ^{2}\right) ^{p/2}\left(
\sum_{k=1}^{n}\left\vert y_{k}\right\vert ^{2}\right) ^{p/2}-\left\vert
\sum_{k=1}^{n}x_{k}\overline{y}_{k}\right\vert ^{p}  \label{e.3.1} \\
\geq \left( \det \left[ 
\begin{array}{ccc}
\left( \sum_{k=1}^{n}\left\vert x_{k}\right\vert ^{2}\right) ^{1/2} &  & 
\left( \left( \sum_{k=1}^{n}\left\vert x_{k}\right\vert ^{2}\right)
^{p/2}-\left\vert \sum_{k=1}^{n}x_{k}\overline{e}_{k}\right\vert ^{p}\right)
^{1/p} \\ 
&  &  \\ 
\left( \sum_{k=1}^{n}\left\vert y_{k}\right\vert ^{2}\right) ^{1/2} &  & 
\left( \left( \sum_{k=1}^{n}\left\vert y_{k}\right\vert ^{2}\right)
^{p/2}-\left\vert \sum_{k=1}^{n}y_{k}\overline{e}_{k}\right\vert ^{p}\right)
^{1/p}%
\end{array}%
\right] \right) ^{p},
\end{multline}%
and%
\begin{multline}
\sum_{k=1}^{n}\left\vert x_{k}\right\vert ^{2}\sum_{k=1}^{n}\left\vert
y_{k}\right\vert ^{2}-\left\vert \sum_{k=1}^{n}x_{k}\overline{y}%
_{k}\right\vert ^{2}  \label{e.3.2} \\
\geq \left( \det \left[ 
\begin{array}{ccc}
\left\vert \sum_{k=1}^{n}x_{k}\overline{e}_{k}\right\vert &  & \left(
\sum_{k=1}^{n}\left\vert x_{k}\right\vert ^{2}-\left\vert \sum_{k=1}^{n}x_{k}%
\overline{e}_{k}\right\vert ^{2}\right) ^{1/2} \\ 
&  &  \\ 
\left\vert \sum_{k=1}^{n}y_{k}\overline{e}_{k}\right\vert &  & \left(
\sum_{k=1}^{n}\left\vert y_{k}\right\vert ^{2}-\left\vert \sum_{k=1}^{n}y_{k}%
\overline{e}_{k}\right\vert ^{2}\right) ^{1/2}%
\end{array}%
\right] \right) ^{2}.
\end{multline}

If we take $e_{m}=1$ for $m\in \left\{ 1,...,n\right\} $ and $e_{k}=0$ for
any $k\in \left\{ 1,...,n\right\} ,$ $k\neq m,$ then $\sum_{k=1}^{n}\left%
\vert e_{k}\right\vert ^{2}=1$ and by (\ref{e.3.1}) and (\ref{e.3.2}) we get%
\begin{multline}
\left( \sum_{k=1}^{n}\left\vert x_{k}\right\vert ^{2}\right) ^{p/2}\left(
\sum_{k=1}^{n}\left\vert y_{k}\right\vert ^{2}\right) ^{p/2}-\left\vert
\sum_{k=1}^{n}x_{k}\overline{y}_{k}\right\vert ^{p}  \label{e.3.3} \\
\geq \max_{m\in \left\{ 1,...,n\right\} }\left( \det \left[ 
\begin{array}{ccc}
\left( \sum_{k=1}^{n}\left\vert x_{k}\right\vert ^{2}\right) ^{1/2} &  & 
\left( \left( \sum_{k=1}^{n}\left\vert x_{k}\right\vert ^{2}\right)
^{p/2}-\left\vert x_{m}\right\vert ^{p}\right) ^{1/p} \\ 
&  &  \\ 
\left( \sum_{k=1}^{n}\left\vert y_{k}\right\vert ^{2}\right) ^{1/2} &  & 
\left( \left( \sum_{k=1}^{n}\left\vert y_{k}\right\vert ^{2}\right)
^{p/2}-\left\vert y_{m}\right\vert ^{p}\right) ^{1/p}%
\end{array}%
\right] \right) ^{p},
\end{multline}%
and%
\begin{multline}
\sum_{k=1}^{n}\left\vert x_{k}\right\vert ^{2}\sum_{k=1}^{n}\left\vert
y_{k}\right\vert ^{2}-\left\vert \sum_{k=1}^{n}x_{k}\overline{y}%
_{k}\right\vert ^{2}  \label{e.3.4} \\
\geq \max_{m\in \left\{ 1,...,n\right\} }\left( \det \left[ 
\begin{array}{ccc}
\left\vert x_{m}\right\vert &  & \left( \sum_{1\leq k\neq m\leq n}\left\vert
x_{k}\right\vert ^{2}\right) ^{1/2} \\ 
&  &  \\ 
\left\vert y_{m}\right\vert &  & \left( \sum_{1\leq k\neq m\leq n}\left\vert
y_{k}\right\vert ^{2}\right) ^{1/2}%
\end{array}%
\right] \right) ^{2}.
\end{multline}

For $p=2$ we get from (\ref{e.3.3}) the simpler inequality%
\begin{multline}
\sum_{k=1}^{n}\left\vert x_{k}\right\vert ^{2}\sum_{k=1}^{n}\left\vert
y_{k}\right\vert ^{2}-\left\vert \sum_{k=1}^{n}x_{k}\overline{y}%
_{k}\right\vert ^{2}  \label{e.3.5} \\
\geq \max_{m\in \left\{ 1,...,n\right\} }\left( \det \left[ 
\begin{array}{ccc}
\left( \sum_{k=1}^{n}\left\vert x_{k}\right\vert ^{2}\right) ^{1/2} &  & 
\left( \sum_{1\leq k\neq m\leq n}\left\vert x_{k}\right\vert ^{2}\right)
^{1/2} \\ 
&  &  \\ 
\left( \sum_{k=1}^{n}\left\vert y_{k}\right\vert ^{2}\right) ^{1/2} &  & 
\left( \sum_{1\leq k\neq m\leq n}\left\vert y_{k}\right\vert ^{2}\right)
^{1/2}%
\end{array}%
\right] \right) ^{2}.
\end{multline}

If we take $e_{k}=\frac{1}{\sqrt{n}}$ for $k\in \left\{ 1,...,n\right\} $,
then $\sum_{k=1}^{n}\left\vert e_{k}\right\vert ^{2}=1$ and by (\ref{e.3.1})
and (\ref{e.3.2}) we get%
\begin{multline}
\left( \sum_{k=1}^{n}\left\vert x_{k}\right\vert ^{2}\right) ^{p/2}\left(
\sum_{k=1}^{n}\left\vert y_{k}\right\vert ^{2}\right) ^{p/2}-\left\vert
\sum_{k=1}^{n}x_{k}\overline{y}_{k}\right\vert ^{p}  \label{e.3.6} \\
\geq n^{p}\left( \det \left[ 
\begin{array}{ccc}
\left( \frac{1}{n}\sum_{k=1}^{n}\left\vert x_{k}\right\vert ^{2}\right)
^{1/2} &  & \left( \left( \frac{1}{n}\sum_{k=1}^{n}\left\vert
x_{k}\right\vert ^{2}\right) ^{p/2}-\left\vert \frac{1}{n}%
\sum_{k=1}^{n}x_{k}\right\vert ^{p}\right) ^{1/p} \\ 
&  &  \\ 
\left( \frac{1}{n}\sum_{k=1}^{n}\left\vert y_{k}\right\vert ^{2}\right)
^{1/2} &  & \left( \left( \frac{1}{n}\sum_{k=1}^{n}\left\vert
y_{k}\right\vert ^{2}\right) ^{p/2}-\left\vert \frac{1}{n}%
\sum_{k=1}^{n}y_{k}\right\vert ^{p}\right) ^{1/p}%
\end{array}%
\right] \right) ^{p}
\end{multline}%
and%
\begin{multline}
\sum_{k=1}^{n}\left\vert x_{k}\right\vert ^{2}\sum_{k=1}^{n}\left\vert
y_{k}\right\vert ^{2}-\left\vert \sum_{k=1}^{n}x_{k}\overline{y}%
_{k}\right\vert ^{2}  \label{e.3.7} \\
\geq n^{2}\left( \det \left[ 
\begin{array}{ccc}
\left\vert \frac{1}{n}\sum_{k=1}^{n}x_{k}\right\vert &  & \left( \frac{1}{n}%
\sum_{k=1}^{n}\left\vert x_{k}\right\vert ^{2}-\left\vert \frac{1}{n}%
\sum_{k=1}^{n}x_{k}\right\vert ^{2}\right) ^{1/2} \\ 
&  &  \\ 
\left\vert \frac{1}{n}\sum_{k=1}^{n}y_{k}\right\vert &  & \left( \frac{1}{n}%
\sum_{k=1}^{n}\left\vert y_{k}\right\vert ^{2}-\left\vert \frac{1}{n}%
\sum_{k=1}^{n}y_{k}\right\vert ^{2}\right) ^{1/2}%
\end{array}%
\right] \right) ^{2}.
\end{multline}

The inequality (\ref{e.3.7}) has been obtained recently for real numbers by
S. G. Walker in \cite{Wa}, where some interesting applications for the
celebrated Cramer-Rao inequality are provided as well.

\end{document}